\documentclass[11pt,reqno]{amsart}

\usepackage[T1]{fontenc}
\usepackage[utf8]{inputenc}
\usepackage{cmap}
\usepackage{lmodern}
\input{glyphtounicode}
\usepackage{amsmath,amssymb,amsthm}
\usepackage{mathtools}
\usepackage{booktabs}
\usepackage{tikz}
\usepackage[margin=1.15in]{geometry}
\usepackage[colorlinks=true,linkcolor=blue,citecolor=blue,urlcolor=blue]{hyperref}
\hypersetup{pdftitle={The density of long runs of non-r-free integers}}

\theoremstyle{plain}
\newtheorem{theorem}{Theorem}[section]
\newtheorem{proposition}[theorem]{Proposition}
\newtheorem{lemma}[theorem]{Lemma}
\newtheorem{corollary}[theorem]{Corollary}
\theoremstyle{definition}

\newtheorem{question}[theorem]{Question}
\newtheorem{remark}[theorem]{Remark}
\newtheorem*{definition*}{Definition}

\numberwithin{equation}{section}

\newcommand{\N}{\mathbb{N}}
\newcommand{\Z}{\mathbb{Z}}
\DeclareMathOperator{\dens}{dens}

\begin{document}

\title[The density of long runs of non-$r$-free integers]
{The density of long runs of non-$r$-free integers}

\author{Xinyang Jiang}
\email{jiangxinyang222@gmail.com}


\begin{abstract}
Fix $r\ge2$ and let $\Delta^{(r)}_k$ be the natural density of the integers $n$ for which
none of $n+1,\dots,n+k$ is $r$-free. We prove that
\[
  \log\frac1{\Delta^{(r)}_k}
  =\frac{r-1}{\zeta(r)}\,k\log k+\frac{r}{\zeta(r)}\,k\log\log k-C_r\,k+o(k),
\]
\[
  C_r=\frac1{\zeta(r)}\Bigl((r-1)\log\zeta(r)+r-1+r\log\frac{\pi}{r\sin(\pi/r)}\Bigr).
\]
The leading term is due to Grimmett; the other two terms are new. For $r=2$,
$C_2=\frac{6}{\pi^{2}}\bigl(1+\log\frac{\pi^{4}}{24}\bigr)=1.4595513\ldots$, and the same
expansion holds for $\log(1/\alpha(k+1))$, where $\alpha(h)$ is Mirsky's density of gaps of
length $h$ between consecutive squarefree numbers. Both bounds reduce to a sum
$m!\,e_m$ over injections into the primes above $k^{1/r}$, which is evaluated by a saddle
point; the lower bound uses the Chung--Erd\H{o}s inequality. As a consequence we obtain the
second term in the first-moment threshold for gaps between $r$-free numbers. For $r=2$ its
leading constant is $\pi^{2}/6$, twice the constant $\pi^{2}/12$ in Erd\H{o}s's lower bound
for the largest gap between squarefree numbers below $x$, and we show that every covering
system of length $k$ has modulus at least $\exp((r/\zeta(r)+o(1))k\log k)$, so that bounding
the first solution of a covering system by its modulus cannot give a constant larger than
$\pi^{2}/12$.
\end{abstract}

\maketitle

\section{Introduction}\label{sec:intro}

Let $r\ge2$. An integer is \emph{$r$-free} if it is divisible by no $r$-th power of a
prime; the $r$-free integers have density $1/\zeta(r)$. Write $s_1<s_2<\cdots$ for the
$r$-free integers, suppressing $r$ from the notation. For $r=2$, Erd\H{o}s observed that
for infinitely many $i$
\begin{equation}\label{eq:erdos-lower}
  s_{i+1}-s_i>(1+o(1))\,\frac{\pi^{2}}{12}\,\frac{\log s_i}{\log\log s_i}.
\end{equation}
This is the form in which he states it in \cite[(3)]{Erdos1966}, with a reference to
\cite{Erdos1951}; it follows from the Chinese remainder theorem, and a proof is sketched by
Grimmett \cite[Appendix~1]{Grimmett}. In \cite[(20)]{Erdos1951} the same inequality is
printed, without a detailed proof, with $\pi^{2}/6$ in place of $\pi^{2}/12$; as Grimmett
points out, the argument gives only $\pi^{2}/12$, and with the constant $\pi^{2}/6$ the
inequality is not known (see \S\ref{sec:constructive} and Question~\ref{q:erdos20}). Erd\H{o}s added in
\cite{Erdos1951} that he did not know whether the inequality had been published before but
that it ``was certainly known to several mathematicians, e.g.\ Bateman, Chowla and
Mirsky''. He continued: ``The curious thing is that it seems to be extremely hard to
replace $\pi^{2}/6$ by any larger constant, in fact it seems possible that for $i>i_0$
\begin{equation}\label{eq:erdos-upper}
  s_{i+1}-s_i<(1+\varepsilon)\,\frac{\pi^{2}}{6}\,\frac{\log s_i}{\log\log s_i}.\text{''}
\end{equation}
Since the constant of the lower bound was misprinted there, it is not clear whether
\eqref{eq:erdos-upper} was intended with $\pi^{2}/6$ or with $\pi^{2}/12$: Grimmett
\cite{Grimmett} states it with $\pi^{2}/12$, while Erd\H{o}s Problem \#208 in \cite{Bloom}
states both inequalities with $\pi^{2}/6$. Corollary~\ref{cor:threshold} below indicates
that $\pi^{2}/6$ is the natural constant in \eqref{eq:erdos-upper}; see the discussion
following it. In either form \eqref{eq:erdos-upper} is open, and even the much weaker
assertion $s_{i+1}-s_i\ll_\varepsilon s_i^{\varepsilon}$ is open, the best unconditional
exponent being $1/5-\eta$ for a small $\eta>0$ by Pandey \cite{Pandey}, after Filaseta and
Trifonov \cite{FT1992}. Granville \cite{Granville} showed that the $abc$ conjecture implies
$s_{i+1}-s_i\ll_\varepsilon s_i^{\varepsilon}$. For a recent study of the closely related
question of which sequences $A$ have infinitely many translates $n+A$ inside the squarefree
numbers, see van Doorn and Tao \cite{VDT}.

This paper is about the local quantity underlying \eqref{eq:erdos-lower}, for every $r$.
Put
\begin{equation}\label{eq:def-D}
  D^{(r)}_k=\{n\in\N:\ \text{none of } n+1,\dots,n+k\ \text{is } r\text{-free}\},\qquad
  \Delta^{(r)}_k=\dens D^{(r)}_k .
\end{equation}
The density exists (\S\ref{sec:prelim}), and $\Delta^{(r)}_k>0$ for every $k$ by the
Chinese remainder theorem. Mirsky \cite[Thm.~4]{Mirsky} proved that for every $h\ge1$ the
density
\begin{equation}\label{eq:def-alpha}
  \alpha_r(h)=\dens\{n:\ n\text{ is }r\text{-free},\ n+1,\dots,n+h-1\text{ are not},\
  n+h\text{ is}\}
\end{equation}
exists; we write $\alpha=\alpha_2$. The two families are related by
$\Delta^{(r)}_k=\sum_{h>k}(h-k)\alpha_r(h)$, and Theorem~\ref{thm:main} is proved
directly for both.

Estimates for $\alpha(h)$ enter the moment problem
\begin{equation}\label{eq:moments}
  \sum_{s_{i+1}\le x}(s_{i+1}-s_i)^{\gamma}\sim B(\gamma)\,x,
  \qquad B(\gamma)=\sum_{h\ge1}h^{\gamma}\alpha(h).
\end{equation}
This is Erd\H{o}s Problem \#145 in \cite{Bloom}, proved by
Erd\H{o}s \cite{Erdos1951} for $\gamma\le2$, by Hooley \cite{Hooley} for $\gamma\le3$, and,
after work of Filaseta \cite{Filaseta1993}, Filaseta--Trifonov \cite{FT1996} and Huxley
\cite{Huxley1997,Huxley2000}, most recently by Chan \cite{Chan} for $\gamma<3.75$. The
estimate for $\alpha(h)$ used in that literature (see \cite[(2)]{Chan}, attributed to
\cite[Lemma~1]{Huxley1997}) is
\begin{equation}\label{eq:huxley-lemma}
  \log\alpha(h)\le-\tfrac54\,h\log\log h+O(h).
\end{equation}
The correct order of magnitude was found by Grimmett \cite{Grimmett}, who proved, by
studying large deviations in a ``random sieve'', that
\begin{equation}\label{eq:grimmett}
  \log\Delta^{(r)}_k=-(1+o(1))\,\frac{r-1}{\zeta(r)}\,k\log k,\qquad
  \log\alpha_r(k+1)=-(1+o(1))\,\frac{r-1}{\zeta(r)}\,k\log k .
\end{equation}
Grimmett works with an arbitrary increasing sequence of pairwise coprime moduli
$d_1<d_2<\cdots$ with $\sum_j1/d_j<\infty$, subject to two regularity conditions on its
counting function, and the $r$-free integers are the case $d_j=P_j^{r}$, $P_j$ the $j$-th
prime; for $r=2$ this is his main theorem. He remarks that in the $r$-free case one may in principle obtain further
information about the $o(1)$ in \eqref{eq:grimmett} from finer properties of the
distribution of the primes. Our main result does this: it determines the next two terms of
the expansion, the second with an explicit constant, for every $r$.

\begin{theorem}\label{thm:main}
Fix $r\ge2$ and put $q_r=1/\zeta(r)$ and
\begin{equation}\label{eq:Cr}
  C_r=q_r\Bigl((r-1)\log\zeta(r)+r-1+r\log\frac{\pi}{r\sin(\pi/r)}\Bigr).
\end{equation}
Then, as $k\to\infty$,
\[
  \log\frac1{\Delta^{(r)}_k}
  =(r-1)q_r\,k\log k+r\,q_r\,k\log\log k-C_r\,k+o(k).
\]
The same asymptotic holds for $\log(1/\alpha_r(k+1))$.
\end{theorem}

For $r=2$ the constant simplifies to
\begin{equation}\label{eq:C2}
  C_2=\frac{6}{\pi^{2}}\Bigl(1+\log\frac{\pi^{4}}{24}\Bigr)=1.45955133487\ldots,
\end{equation}
and Theorem~\ref{thm:main} reads
$\log(1/\Delta^{(2)}_k)=\frac{6}{\pi^{2}}k\log k+\frac{12}{\pi^{2}}k\log\log k-C_2k+o(k)$.
Table~\ref{tab:Cr} lists the constants for small $r$.

\begin{table}[ht]
\caption{The three constants of Theorem~\ref{thm:main}.}
\label{tab:Cr}
\begin{tabular}{rllll}
\toprule
$r$ & $q_r=1/\zeta(r)$ & $(r-1)q_r$ & $r\,q_r$ & $C_r$\\
\midrule
$2$ & $0.6079271019$ & $0.6079271019$ & $1.215854204$ & $1.45955133487$\\
$3$ & $0.8319073726$ & $1.663814745$ & $2.495722118$ & $2.44409748264$\\
$4$ & $0.9239384029$ & $2.771815209$ & $3.695753612$ & $3.37918097415$\\
$5$ & $0.9643873404$ & $3.857549362$ & $4.821936702$ & $4.31898650447$\\
$6$ & $0.9829525923$ & $4.914762961$ & $5.897715554$ & $5.27125777085$\\
\bottomrule
\end{tabular}
\end{table}

The leading term of Theorem~\ref{thm:main} is Grimmett's \eqref{eq:grimmett}. (For $r=2$
the lower bound $\Delta^{(2)}_k\ge\exp(-(1+o(1))q_2k\log k)$ is essentially the
Chinese-remainder construction behind \eqref{eq:erdos-lower} combined with a count of the
assignments of primes to positions, as sketched in \cite[Appendix~1]{Grimmett}.) The new
terms are $r\,q_r\,k\log\log k$ and $-C_rk$, so the error is $o(k)$ instead of
$o(k\log k)$. Erd\H{o}s's paper \cite{Erdos1951} gives no estimate for $\alpha(h)$: his
Lemma~1 records the existence of the density, citing \cite{Mirsky}, and his Lemma~2 is the
uniform bound $\sum_{h>t}\#\{i:s_{i+1}\le x,\ s_{i+1}-s_i=h\}\ll x/(t^{2}\log^{2}t)$, which
is what his proof of \eqref{eq:moments} for $\gamma\le2$ uses.

\begin{remark}\label{rem:nogain}
Neither \eqref{eq:grimmett} nor Theorem~\ref{thm:main} enlarges the range of $\gamma$ in
\eqref{eq:moments}. There $\alpha(h)$ enters only through the convergence of $B(\gamma)$,
and \eqref{eq:huxley-lemma} already gives this for all $\gamma\ge0$. The obstruction to $\gamma\ge3.75$ lies in the lattice
point counts discussed in Appendix~\ref{app:moment}, where at the critical configuration a
gain in any one of the exponents involved would already enlarge the range.
\end{remark}

The lower-order terms of Theorem~\ref{thm:main} are what is needed for the following
refinement of the first-moment threshold.

\begin{corollary}\label{cor:threshold}
Let $K_r(x)=\max\{k:\Delta^{(r)}_k\ge1/x\}$ be the first-moment threshold. Then
\[
  K_r(x)=\frac{\zeta(r)}{r-1}\cdot\frac{\log x}{\log\log x}\cdot
  \Bigl(1-\frac{1+o(1)}{r-1}\cdot\frac{\log\log\log x}{\log\log x}\Bigr).
\]
\end{corollary}

The leading term $K_r(x)\sim\frac{\zeta(r)}{r-1}\log x/\log\log x$ already follows from
\eqref{eq:grimmett}; the correction factor requires the second-order term of
Theorem~\ref{thm:main}. For $r=2$ the leading constant is $\pi^{2}/6$, twice the constant
$\pi^{2}/12$ of \eqref{eq:erdos-lower}. At this constant the expected number of runs of
length $k$ below $x$ passes from $\gg1$ to $\ll1$. So no first-moment argument gives a lower
bound for the largest gap with a constant larger than $\pi^{2}/6$; and
\eqref{eq:erdos-upper} with a smaller constant, for instance $\pi^{2}/12$, would mean that
runs are spread very unevenly, since the expected number below $x$ of runs of length
$(\pi^{2}/6-\delta)\log x/\log\log x$ is $x^{c(\delta)+o(1)}$ with $c(\delta)>0$. Grimmett
\cite{Grimmett} arrives at $\pi^{2}/6$ in \eqref{eq:erdos-upper} from a heuristic for the
moments $\sum_{s_{i+1}\le x}g(s_{i+1}-s_i)$ with $g(h)=h^{\theta h}$, and at $\pi^{2}/6$ in
the lower bound by assuming that the runs produced by the Chinese remainder theorem are
uniformly spread (see \S\ref{sec:constructive}). We therefore take $\pi^{2}/6$ to be the
natural constant in \eqref{eq:erdos-upper}. Corollary~\ref{cor:threshold} also shows that
the threshold lies below $\zeta(2)\log x/\log\log x$, by a relative amount
$(1+o(1))\log\log\log x/\log\log x$. None of this bounds an individual gap from above, and
\eqref{eq:erdos-upper} remains open.

\subsection*{Outline of the proof}
Apart from the prime number theorem the proof is elementary. The main input is
Lemma~\ref{lem:sieve}: if $z$ is so small that $Q=\prod_{p\le z}p^{r}\le\sqrt k$, then the
number of $i\in[1,k]$ surviving the sieve by the $r$-th powers $p^{r}$, $p\le z$, is
$k\prod_{p\le z}(1-p^{-r})+O(Q)$ for every $n$, because the window contains many complete
periods modulo $Q$. Since
$\prod_{p\le z}(1-p^{-r})\to q_r$, at least $(q_r-o(1))k$ positions of the window must be
covered by $r$-th powers of primes exceeding $k^{1/r}$, and such primes are necessarily
pairwise distinct because their $r$-th powers exceed $k$. Everything reduces to estimating
$m!\,e_m$, where $e_m$ is the $m$-th elementary symmetric function of
$\{p^{-r}:p>k^{1/r}\}$ and $m\approx q_rk$.

Grimmett's large-deviation analysis of a random sieve \cite{Grimmett} determines
$m!\,e_m$ up to a factor $e^{o(k\log k)}$; here both bounds determine it up to a factor
$e^{o(k)}$. The upper bound
(\S\ref{sec:upper}) is a Chernoff estimate
whose saddle point is governed by $\int_0^{\infty}(1+v^{r})^{-1}dv=(\pi/r)/\sin(\pi/r)$.
For the lower bound (\S\ref{sec:lower}) one cannot simply exhibit disjoint congruence
classes attached to the $m$ cheapest primes: that recovers only the largest term of $e_m$
and loses a factor $\exp((r\log\frac{\pi}{r\sin(\pi/r)}+o(1))m)$, which accounts for the
term $r\log\frac{\pi}{r\sin(\pi/r)}$ in $C_r$. Instead we apply the Chung--Erd\H{o}s
inequality to the family of all injections into the primes above $k^{1/r}$; the pairwise
intersections are controlled by $\sum_{p>k^{1/r}}p^{-r}\ll k^{1/r-1}/\log k$, so the loss
is only $e^{o(k)}$.

In \S\ref{sec:numerics} we compute $\Delta^{(2)}_k$ and $\alpha(h)$ exactly for
$k,h\le150$ and $\Delta^{(3)}_k$ for $k\le120$, compare them with the coefficient $r\,q_r$
and the constant $C_r$, compute $B(\gamma)$, and estimate the first occurrence of a run of
length $k$ for $19\le k\le32$ (sequence \texttt{A045882} of \cite{OEIS} is known up to
$k=18$). Appendix~\ref{app:moment} concerns the exponent in the Huxley--Chan treatment of
\eqref{eq:moments}.

\section{Notation and preliminaries}\label{sec:prelim}

Throughout, $r\ge2$ is fixed, $p$ denotes a prime, $\pi(y)=\#\{p\le y\}$,
$\theta(y)=\sum_{p\le y}\log p$, $P_j$ is the $j$-th prime, and $q_r=1/\zeta(r)=
\prod_p(1-p^{-r})$. For a set $A$ of non-negative reals with $\sum_{a\in A}a<\infty$ we
write $e_m(A)$ for its $m$-th elementary symmetric function. Implied constants may depend
on $r$ but on nothing else.

\subsection*{Existence of the densities}
Fix $k$. For $y\ge2$ let $D_k(y)$ be the set of $n$ such that each of $n+1,\dots,n+k$ is
divisible by $p^{r}$ for some $p\le y$. Each $D_k(y)$ is a union of congruence classes
modulo $\prod_{p\le y}p^{r}$, hence has a density $\delta_k(y)$, non-decreasing in $y$ and
bounded by $1$. Since
$0\le\overline{\dens}\,(D^{(r)}_k\setminus D_k(y))\le k\sum_{p>y}p^{-r}\to0$, the density
$\Delta^{(r)}_k=\lim_{y\to\infty}\delta_k(y)$ exists. The same argument gives the existence
of $\alpha_r(h)$, which is also \cite[Thm.~4]{Mirsky}.

\subsection*{Relation between $\Delta^{(r)}_k$ and $\alpha_r(h)$}
A maximal run of exactly $\ell$ consecutive non-$r$-free integers contains $\ell-k+1$
starting points of runs of length $k$ when $\ell\ge k$, and a gap $s_{i+1}-s_i=h$
corresponds to a maximal run of length $h-1$. Hence
\begin{equation}\label{eq:D-alpha}
  \Delta^{(r)}_k=\sum_{h>k}(h-k)\,\alpha_r(h),\qquad\text{so}\qquad
  \alpha_r(k+1)=\Delta^{(r)}_k-2\Delta^{(r)}_{k+1}+\Delta^{(r)}_{k+2}.
\end{equation}
In particular
\begin{equation}\label{eq:alpha-Delta}
  \alpha_r(k+1)\le\Delta^{(r)}_k ,
\end{equation}
so the upper bound of
Theorem~\ref{thm:main} for $\alpha_r(k+1)$ follows from that for $\Delta^{(r)}_k$; the
matching lower bound for $\alpha_r(k+1)$ is proved in \S\ref{sec:lower} by the same method
as for $\Delta^{(r)}_k$, requiring in addition that $n$ and $n+k+1$ be $r$-free.
(Proceeding via \eqref{eq:D-alpha} alone would require a bound on
$\Delta^{(r)}_{k+1}/\Delta^{(r)}_k$, which does not follow from an asymptotic for
$\log\Delta^{(r)}_k$ with error $o(k)$; as pointed out in \cite[Appendix~B]{Cellarosi},
an earlier version of this paper deduced $\alpha_r(k+1)=(1+o(1))\Delta^{(r)}_k$ in this
way, and moreover misprinted the coefficient in the ratio as $r-1$ in place of
$(r-1)q_r$.)

We shall use four lemmas.

\begin{lemma}[deterministic sieve count]\label{lem:sieve}
Let $z\ge2$ and $Q=\prod_{p\le z}p^{r}$. Then for every $n\in\Z$ and every $k\ge1$,
\[
  \#\{1\le i\le k:\ p^{r}\nmid n+i\ \text{ for all } p\le z\}
  \;=\;k\prod_{p\le z}\Bigl(1-\frac1{p^{r}}\Bigr)+O(Q).
\]
\end{lemma}

\begin{proof}
The condition on $i$ depends only on $i\bmod Q$, and by the Chinese remainder theorem
exactly $Q\prod_{p\le z}(1-p^{-r})$ residues modulo $Q$ satisfy it. The interval $[1,k]$
contains $\lfloor k/Q\rfloor$ complete residue systems modulo $Q$ and a remainder of length
less than $Q$.
\end{proof}

\begin{lemma}[Chernoff bound for $e_m$]\label{lem:chernoff}
For any $t>0$, $\;e_m(A)\le t^{-m}\prod_{a\in A}(1+ta)$.
\end{lemma}

\begin{proof}
$\prod_{a\in A}(1+ta)=\sum_{j\ge0}t^{j}e_j(A)\ge t^{m}e_m(A)$.
\end{proof}

\begin{lemma}[a sum over primes]\label{lem:euler}
As $s\to\infty$,
\[
  F_r(s):=\sum_{p}\log\Bigl(1+\frac{s^{r}}{p^{r}}\Bigr)
  =\frac{\pi}{\sin(\pi/r)}\cdot\frac{s}{\log s}+O\Bigl(\frac{s}{\log^{2}s}\Bigr).
\]
Moreover, for $2\le w\le s$, $\;0\le F_r(s)-\sum_{p>w}\log(1+s^{r}/p^{r})\ll\pi(w)\log s$.
\end{lemma}

\begin{proof}
By partial summation, $F_r(s)=\int_{2}^{\infty}\pi(u)\,\frac{r\,s^{r}}{u(u^{r}+s^{r})}\,du$,
the boundary terms vanishing. Inserting $\pi(u)=u/\log u+O(u/\log^{2}u)$ and substituting
$u=sv$,
\[
  F_r(s)=r\,s\int_{2/s}^{\infty}\frac{dv}{(1+v^{r})\log(sv)}
  +O\Bigl(s\int_{2/s}^{\infty}\frac{dv}{(1+v^{r})\log^{2}(sv)}\Bigr).
\]
Writing $\log(sv)=\log s\,(1+\log v/\log s)$ and using the classical evaluation
$\int_0^{\infty}(1+v^{r})^{-1}dv=(\pi/r)/\sin(\pi/r)$ together with the convergence of
$\int_0^{\infty}|\log v|(1+v^{r})^{-1}dv$ yields the stated asymptotic; the range
$v<2/s$ contributes $O(1)$ because $\log(sv)\ge\log2$ there. For the second claim,
$0\le\log(1+s^{r}/p^{r})\le r\log s+\log2$ for $p\ge2$, so removing the primes $p\le w$
costs at most $\pi(w)(r\log s+\log2)$.
\end{proof}

Given $m$ large, let $s=s(m)$ be the unique solution of
\begin{equation}\label{eq:saddle}
  \frac{\pi}{\sin(\pi/r)}\cdot\frac{s}{\log s}=r\,m ,
\end{equation}
so that $s=\frac{r\sin(\pi/r)}{\pi}\,m\log s$ and, iterating,
\begin{equation}\label{eq:logs}
  \log s=\log m+\log\log m+\log\frac{r\sin(\pi/r)}{\pi}+O\Bigl(\frac{\log\log m}{\log m}\Bigr).
\end{equation}

The last lemma shows that Lemma~\ref{lem:chernoff} is sharp up to a factor $O(\sqrt m)$. It
uses Darroch's theorem \cite{Darroch}: if $X$ is a sum of independent Bernoulli variables and
$\mathbb E X$ is an integer, then $\mathbb E X$ is a mode of $X$.

\begin{lemma}[the Chernoff bound is tight]\label{lem:local}
Let $A=\{a_1,a_2,\dots\}$ be non-negative with $\sum a_j<\infty$ and let $m\ge1$ be an
integer for which some $t>0$ satisfies $\sum_j\frac{ta_j}{1+ta_j}=m$. Then
\[
  t^{-m}\prod_j(1+ta_j)\ \ge\ e_m(A)\ \ge\ \frac{c}{\sqrt m}\,t^{-m}\prod_j(1+ta_j)
\]
for an absolute constant $c>0$.
\end{lemma}

\begin{proof}
Let $X=\sum_jB_j$ with $B_j$ independent Bernoulli of parameter
$\theta_j=ta_j/(1+ta_j)$. Expanding $\prod_j(1-\theta_j+\theta_j)$ shows
$\mathbb P(X=m)=t^{m}e_m(A)\prod_j(1+ta_j)^{-1}$, which gives the upper bound
(Lemma~\ref{lem:chernoff}) and turns the lower bound into a statement about
$\mathbb P(X=m)$. By hypothesis $\mathbb E X=m$ is an integer. For a finite sum $Y$ of independent Bernoulli
variables with $\mathbb E Y=m$, Darroch's theorem implies that $m$ is a mode of $Y$, and its
variance is $\sigma_Y^{2}\le m$, so by Chebyshev at least $3/4$ of the mass of $Y$ lies on
the at most $4\sigma_Y+1$ integers within $2\sigma_Y$ of $m$; hence
$\mathbb P(Y=m)=\max_j\mathbb P(Y=j)\ge\frac34(4\sigma_Y+1)^{-1}\ge c/\sqrt m$. In the
infinite case, let $\delta_N=\sum_{j>N}\theta_j\to0$ and, for $N$ large enough that
$\delta_N<1$, put $X_N=\sum_{j=1}^{N}B_j+D_N$ where $D_N$ is an independent Bernoulli
variable of parameter $\delta_N$. Then $\mathbb E X_N=m$, so $\mathbb P(X_N=m)\ge c/\sqrt m$,
while $\mathbb P(X_N\ne X)\le\mathbb P(D_N=1)+\sum_{j>N}\mathbb P(B_j=1)=2\delta_N\to0$;
letting $N\to\infty$ gives $\mathbb P(X=m)\ge c/\sqrt m$.
\end{proof}

\section{The upper bound}\label{sec:upper}

\begin{proposition}\label{prop:upper}
$\log\Delta^{(r)}_k\le-(r-1)q_r\,k\log k-r\,q_r\,k\log\log k+C_r\,k+o(k)$.
\end{proposition}

\begin{proof}
Let $k$ be large and choose $z=z(k)$ maximal with $\theta(z)\le\frac1{2r}\log k$; then
$z\asymp\log k$ and $Q:=\prod_{p\le z}p^{r}=e^{r\theta(z)}\le\sqrt k$.

\smallskip
\emph{Step 1: the surviving set.}
For $n\in\Z$ put $V(n)=\{1\le i\le k:\ p^{r}\nmid n+i\ \forall p\le z\}$. By
Lemma~\ref{lem:sieve} and $Q\le\sqrt k$,
$|V(n)|=k\,q_{r,z}+O(\sqrt k)$ with $q_{r,z}=\prod_{p\le z}(1-p^{-r})$, for every $n$.
Since $q_{r,z}=q_r\prod_{p>z}(1-p^{-r})^{-1}=q_r(1+O(z^{1-r}/\log z))$,
\begin{equation}\label{eq:Vsize}
  |V(n)|=q_r\,k\Bigl(1+O\Bigl(\frac1{\log^{r-1}k\,\log\log k}\Bigr)\Bigr)
  \qquad\text{for every }n .
\end{equation}

\smallskip
\emph{Step 2: removing the middle primes.}
Let $W(n)=\{i\in V(n):\ p^{r}\nmid n+i\ \forall p\le k^{1/r}\}$. The number of
$i\in[1,k]$ divisible by $p^{r}$ for some $p\in(z,k^{1/r}]$ is at most
$\sum_{z<p\le k^{1/r}}(k/p^{r}+1)\ll k/(z^{r-1}\log z)+\pi(k^{1/r})
\ll k/(\log^{r-1}k\,\log\log k)$. Hence, with \eqref{eq:Vsize}, there is a constant $C'$
with
\begin{equation}\label{eq:m1}
  |W(n)|\ \ge\ m_1:=\Bigl\lceil q_r\,k\Bigl(1-\frac{C'}{\log^{r-1}k\,\log\log k}\Bigr)
  \Bigr\rceil\qquad\text{for every }n .
\end{equation}

\smallskip
\emph{Step 3: an injection into large primes.}
Suppose $n\in D^{(r)}_k$. Every $i\in W(n)$ has $p^{r}\mid n+i$ for some prime
$p>k^{1/r}$, and for distinct $i,i'\in W(n)$ the corresponding primes are distinct, since
$p^{r}>k>|i-i'|$ forbids $p^{r}\mid n+i$ and $p^{r}\mid n+i'$ simultaneously.

Put $Q'=\prod_{p\le k^{1/r}}p^{r}$ and note that $V(n),W(n)$ depend only on
$\rho=n\bmod Q'$; write $W(\rho)$ accordingly and fix a subset $W'(\rho)\subseteq W(\rho)$
of size $m_1$. Let $y>k^{1/r}$, $\mathcal P_y=\{p:k^{1/r}<p\le y\}$, and let $\Phi_y(\rho)$ be the set of
injections $\varphi:W'(\rho)\to\mathcal P_y$. For $n\in D_k(y)$ (defined in \S\ref{sec:prelim}), every $i\in W(n)$ is covered by some prime in $\mathcal P_y$, and as above these primes are pairwise distinct. Hence
\[
  D_k(y)\ \subseteq\ \bigcup_{\rho\bmod Q'}\ \bigcup_{\varphi\in\Phi_y(\rho)}
  \bigl\{n\in\N:\ n\equiv \rho\ (Q'),\ \varphi(i)^{r}\mid n+i\ \forall i\in W'(\rho)\bigr\}.
\]
For fixed $y$ both unions are finite, and each inner set is a single congruence class modulo
$Q'\prod_{i\in W'(\rho)}\varphi(i)^{r}$, of density $\bigl(Q'\prod_{i\in W'(\rho)}\varphi(i)^{r}\bigr)^{-1}$.
By subadditivity of the density over a finite union,
\[
  \delta_k(y)\ \le\ \sum_{\rho\bmod Q'}\frac1{Q'}\sum_{\varphi\in\Phi_y(\rho)}\prod_{i\in W'(\rho)}\varphi(i)^{-r}
  \ \le\ m_1!\,e_{m_1},\qquad e_m:=e_m\bigl(\{p^{-r}:p>k^{1/r}\}\bigr),
\]
the sum over injections from a set of size $m_1$ into $\mathcal P=\{p:p>k^{1/r}\}$ being independent of
$\rho$, and $\sum_{\rho}1/Q'=1$. Letting $y\to\infty$ gives
\begin{equation}\label{eq:upper-reduction}
  \Delta^{(r)}_k\ \le\ m_1!\,e_{m_1}.
\end{equation}
(In an earlier version of this proof we took the union over all injections into $\mathcal P$ before passing to the density and claimed that for $n\le x$ the inner set is empty unless $\prod_i\varphi(i)^{r}\le x+k$; as pointed out in \cite[Appendix~B]{Cellarosi}, this fails because the prime powers $\varphi(i)^{r}$ divide distinct translates $n+i$, so one only has $\prod_i\varphi(i)^{r}\le(x+k)^{m_1}$. Working with $D_k(y)$ first and letting $y\to\infty$ avoids this step.)

\smallskip
\emph{Step 4: estimating $m!\,e_m$.}
Write $m=m_1$ and let $s=s(m)$ be as in \eqref{eq:saddle}. By Lemma~\ref{lem:chernoff} with
$t=s^{r}$ and then Lemma~\ref{lem:euler} with $w=k^{1/r}$,
\[
  \log e_m\ \le\ \sum_{p>k^{1/r}}\log\Bigl(1+\frac{s^{r}}{p^{r}}\Bigr)-rm\log s
  \ \le\ \frac{\pi}{\sin(\pi/r)}\frac{s}{\log s}+O\Bigl(\frac{s}{\log^{2}s}\Bigr)
  +O\bigl(k^{1/r}\bigr)-rm\log s .
\]
By \eqref{eq:saddle} the first term equals $rm$, and $s\asymp m\log m\asymp k\log k$ gives
$s/\log^{2}s\ll m/\log m$. With \eqref{eq:logs},
\[
  \log e_m\ \le\ rm-rm\log m-rm\log\log m-rm\log\frac{r\sin(\pi/r)}{\pi}+o(m).
\]
Since $\log m!=m\log m-m+O(\log m)$,
\begin{equation}\label{eq:me}
  \log\bigl(m!\,e_m\bigr)\ \le\ -(r-1)m\log m-rm\log\log m
  +\Bigl(r-1+r\log\frac{\pi}{r\sin(\pi/r)}\Bigr)m+o(m).
\end{equation}
Finally $m=m_1=q_rk(1+O(\log^{1-r}k/\log\log k))$, and since the right side of
\eqref{eq:me} has derivative $-(r-1)\log m+O(\log\log m)$ in $m$, replacing $m_1$ by $q_rk$
costs only $O(k/\log\log k)=o(k)$. Using
$\log(q_rk)=\log k-\log\zeta(r)$ and $\log\log(q_rk)=\log\log k+O(1/\log k)$ turns
\eqref{eq:me} into
\[
  \log\Delta^{(r)}_k\le-(r-1)q_rk\log k-rq_rk\log\log k
  +q_r\Bigl((r-1)\log\zeta(r)+r-1+r\log\frac{\pi}{r\sin(\pi/r)}\Bigr)k+o(k),
\]
which is the assertion.
\end{proof}

\section{The lower bound}\label{sec:lower}

We use the Chung--Erd\H{o}s inequality \cite{ChungErdos}: for events $A_1,\dots,A_M$ in a
probability space with $\sum_i\mathbb P(A_i)>0$,
\begin{equation}\label{eq:CE}
  \mathbb P\Bigl(\bigcup_iA_i\Bigr)\ \ge\
  \frac{\bigl(\sum_i\mathbb P(A_i)\bigr)^{2}}{\sum_{i,j}\mathbb P(A_i\cap A_j)} .
\end{equation}

\begin{proposition}\label{prop:lower}
$\log\Delta^{(r)}_k\ge-(r-1)q_r\,k\log k-r\,q_r\,k\log\log k+C_r\,k-o(k)$, and the same lower bound holds for $\log\alpha_r(k+1)$.
\end{proposition}

\begin{proof}
Let $z$, $Q$ be as in \S\ref{sec:upper}, so $Q\le\sqrt k$ and $\log Q\le\frac12\log k$.
Fix a residue class $\rho_0$ modulo $Q$; to treat $\alpha_r(k+1)$ at the same time, choose
$\rho_0$ so that $p^{r}\nmid\rho_0$ and $p^{r}\nmid\rho_0+k+1$ for every $p\le z$ (which is
possible by the Chinese remainder theorem since $p^{r}\ge4>2$ for every prime $p$). By Lemma~\ref{lem:sieve} the set $V=V(n)$ is the
same for all $n\equiv\rho_0\ (Q)$, and by \eqref{eq:Vsize}
\begin{equation}\label{eq:m-lower}
  m:=|V|=q_r\,k\Bigl(1+O\Bigl(\frac1{\log^{r-1}k\,\log\log k}\Bigr)\Bigr).
\end{equation}
Let $Y>k$ be a parameter, let $\mathcal P_Y=\{p:k^{1/r}<p\le Y\}$, and work in the
probability space $\prod_{p\in\mathcal P_Y}\Z/p^{r}\Z$ with the uniform measure, which by
the Chinese remainder theorem computes densities of sets defined by congruences to the
moduli $p^{r}$, $p\in\mathcal P_Y$, independently of the class $\rho_0$. For an injection
$\varphi:V\to\mathcal P_Y$ set
\[
  A_\varphi=\{n\equiv\rho_0\ (Q):\ \varphi(i)^{r}\mid n+i\text{ for all }i\in V\}.
\]
If $n\in A_\varphi$ then every $i\in[1,k]\setminus V$ has $n+i$ divisible by $p^{r}$ for
some $p\le z$ and every $i\in V$ has $n+i$ divisible by $\varphi(i)^{r}$; hence
$\bigcup_\varphi A_\varphi\subseteq D^{(r)}_k$ and
\begin{equation}\label{eq:lowerQ}
  \Delta^{(r)}_k\ \ge\ \frac1Q\,\mathbb P\Bigl(\bigcup_\varphi A_\varphi\Bigr).
\end{equation}

\emph{The two sums in \eqref{eq:CE}.}
By the Chinese remainder theorem $\mathbb P(A_\varphi)=\prod_{i\in V}\varphi(i)^{-r}$, so
\begin{equation}\label{eq:firstsum}
  \sum_\varphi\mathbb P(A_\varphi)=m!\,e_m^{(Y)},\qquad
  e^{(Y)}_m:=e_m\bigl(\{p^{-r}:p\in\mathcal P_Y\}\bigr).
\end{equation}
For the second sum, let $\varphi,\psi$ be injections and $D=\{i\in V:\varphi(i)\ne\psi(i)\}$.
If some prime is required to divide two of the $n+i$ with distinct $i$, then
$A_\varphi\cap A_\psi=\emptyset$, because $p^{r}>k>|i-i'|$; otherwise the constraints
involve the $m+|D|$ distinct primes $\varphi(V)\cup\psi(D)$ and
$\mathbb P(A_\varphi\cap A_\psi)=\mathbb P(A_\varphi)\prod_{i\in D}\psi(i)^{-r}$. Summing
over $\psi$ by first choosing $D$ and then the values of $\psi$ on $D$,
\begin{equation}\label{eq:secondsum}
  \sum_\psi\mathbb P(A_\varphi\cap A_\psi)\ \le\ \mathbb P(A_\varphi)
  \sum_{d\ge0}\binom{m}{d}\sigma^{d}=\mathbb P(A_\varphi)(1+\sigma)^{m},
  \qquad\sigma:=\sum_{p>k^{1/r}}p^{-r}.
\end{equation}
By the prime number theorem $\sigma\ll k^{(1-r)/r}/\log k$, so
$m\sigma\ll k^{1/r}/\log k=o(k)$ and $(1+\sigma)^{m}=e^{o(k)}$. Combining
\eqref{eq:CE}, \eqref{eq:firstsum} and \eqref{eq:secondsum},
\[
  \mathbb P\Bigl(\bigcup_\varphi A_\varphi\Bigr)\ \ge\
  \frac{m!\,e^{(Y)}_m}{(1+\sigma)^{m}}\ =\ m!\,e^{(Y)}_m\,e^{-o(k)} .
\]
Letting $Y\to\infty$ gives $e^{(Y)}_m\uparrow e_m$, whence by \eqref{eq:lowerQ}
\begin{equation}\label{eq:lower-reduction}
  \log\Delta^{(r)}_k\ \ge\ \log\bigl(m!\,e_m\bigr)-\log Q-o(k)
  \ =\ \log\bigl(m!\,e_m\bigr)-o(k).
\end{equation}

\emph{Evaluating $m!\,e_m$ from below.}
Let $t>0$ solve $\sum_{p>k^{1/r}}\frac{tp^{-r}}{1+tp^{-r}}=m$; such a $t$ exists because the
left side is continuous, increasing, and unbounded. By Lemma~\ref{lem:local},
$\log e_m\ge\sum_{p>k^{1/r}}\log(1+tp^{-r})-m\log t-O(\log m)$. The right side is
$G(t^{1/r})$, where $G(\sigma')=\sum_{p>k^{1/r}}\log(1+\sigma'^{r}/p^{r})-rm\log\sigma'$, so it
suffices to bound $G$ from below on all of $(0,\infty)$. For bounded $\sigma'$ this is trivial,
since $G(\sigma')\ge-rm\log\sigma'$. For large $\sigma'$, Lemma~\ref{lem:euler} gives
$G(\sigma')\ge A\sigma'/\log\sigma'-rm\log\sigma'-O(k^{1/r})$ with $A=(1-o(1))\pi/\sin(\pi/r)$,
and the elementary inequality
\[
  A\frac{\sigma'}{\log\sigma'}-B\log\sigma'\ \ge\ B-B\log s_A-\frac{2B}{\log s_A}
  \qquad(\sigma'>1),\qquad B=A\frac{s_A}{\log s_A},\ \log s_A\ge4,
\]
(put $\sigma'=s_Ae^{y}$ and use $e^{y}\ge1+y+y^{2}/2$ for $y\ge0$ and $e^{y}\ge1+y$ for
$y<0$), applied with $B=rm$, shows that the minimum differs from the value at $s(m)$ by $o(m)$.
With \eqref{eq:saddle} and \eqref{eq:logs} this gives
\[
  \log e_m\ \ge\ rm-rm\log m-rm\log\log m-rm\log\frac{r\sin(\pi/r)}{\pi}-o(m),
\]
so that $\log(m!\,e_m)$ equals the right side of \eqref{eq:me} up to $o(m)$. Feeding this
into \eqref{eq:lower-reduction} and using \eqref{eq:m-lower} as at the end of
\S\ref{sec:upper} completes the proof for $\Delta^{(r)}_k$.

\emph{The lower bound for $\alpha_r(k+1)$.}
For $\alpha_r(k+1)$ we must also ensure that $n$ and $n+k+1$ are $r$-free. By our choice of
$\rho_0$, no $p^{r}$ with $p\le z$ divides $n$ or $n+k+1$ when $n\equiv\rho_0\ (Q)$. Set
$z_1=(k+1)^{1/r}$, and for each prime $p\in(z,z_1]$ fix a residue modulo $p^{r}$ missing
$0$ and $-k-1\bmod p^{r}$. Restricting to the resulting class $\rho_1$ modulo
$Q_1=\prod_{p\le z_1}p^{r}$ removes at most
$\sum_{z<p\le z_1}(k/p^{r}+1)=O(k/(\log^{r-1}k\log\log k))$ elements from $V$, leaving a
subset $V_1\subseteq V$ of size $m_2=q_rk(1+O(\log^{1-r}k/\log\log k))$ and multiplying the
density by $1/Q_1=\exp(-r\theta(z_1))=\exp(-O(k^{1/r}))=e^{-o(k)}$. For every prime
$p>z_1$ we have $p^{r}>k+1$, so the $m_2+2$ residues $-i\bmod p^{r}$ ($i\in V_1$) and
$0,-k-1\bmod p^{r}$ are pairwise distinct. In the probability space
$\prod_{z_1<p\le Y}\Z/p^{r}\Z$ (for $Y>k+1$), condition on the event
$E_Y=\{p^{r}\nmid n\text{ and }p^{r}\nmid n+k+1\ \forall p\in(z_1,Y]\}$, which has
probability $\prod_{z_1<p\le Y}(1-2/p^{r})\ge\prod_{p>z_1}(1-2/p^{r})=1-o(1)$. Under the
conditional measure $\mathbb P(\cdot\mid E_Y)$ the coordinates remain independent, and for
each $p\in(z_1,Y]$ the residue $n\bmod p^{r}$ is uniform on the $p^{r}-2$ allowed classes,
which contain all $-i\bmod p^{r}$ ($i\in V_1$). Applying \eqref{eq:CE} under
$\mathbb P(\cdot\mid E_Y)$ to the events
$A'_\varphi=\{\varphi(i)^{r}\mid n+i\ \forall i\in V_1\}$ for injections
$\varphi:V_1\to\{p:z_1<p\le Y\}$ replaces $p^{-r}$ by $(p^{r}-2)^{-1}\ge p^{-r}$ in
\eqref{eq:firstsum} and \eqref{eq:secondsum} (with
$\sigma'=\sum_{p>z_1}(p^{r}-2)^{-1}\ll k^{(1-r)/r}/\log k$), giving
\[
  \dens(\mathcal B_Y)
  =\frac1{Q_1}\,\mathbb P(E_Y)\,\mathbb P\Bigl(\bigcup_\varphi A'_\varphi\;\Big|\;E_Y\Bigr)
  \ge\frac1{Q_1}\prod_{p>z_1}\Bigl(1-\frac2{p^{r}}\Bigr)
  \cdot\frac{m_2!\,e_{m_2}(\{p^{-r}:z_1<p\le Y\})}{(1+\sigma')^{m_2}},
\]
where $\mathcal B_Y=\{n\in\N:n\equiv\rho_1\ (Q_1),\ n\in E_Y\cap\bigcup_\varphi A'_\varphi\}$.
Every $n\in\mathcal B_Y$ has $n+1,\dots,n+k$ non-$r$-free and fails to contribute to
$\alpha_r(k+1)$ only if $p^{r}\mid n$ or $p^{r}\mid n+k+1$ for some $p>Y$; this exceptional
set has upper density at most $2\sum_{p>Y}p^{-r}\to0$ as $Y\to\infty$ with $k$ fixed.
Letting $Y\to\infty$ therefore yields $\log\alpha_r(k+1)\ge\log(m_2!\,e'_{m_2})-o(k)$, where
$e'_{m}=e_m(\{p^{-r}:p>(k+1)^{1/r}\})$. Evaluating $m_2!\,e'_{m_2}$ as above gives the same
lower bound for $\log\alpha_r(k+1)$.
\end{proof}

Propositions \ref{prop:upper} and \ref{prop:lower} prove Theorem~\ref{thm:main}, since
$\alpha_r(k+1)\le\Delta^{(r)}_k$ gives the upper bound for $\alpha_r(k+1)$.

\begin{proof}[Proof of Corollary~\ref{cor:threshold}]
For $x$ large, $k:=K_r(x)$ satisfies $\Delta^{(r)}_{k+1}<1/x\le\Delta^{(r)}_k$, so
$L:=\log(1/\Delta^{(r)}_k)$ and $L':=\log(1/\Delta^{(r)}_{k+1})$ satisfy
$L\le\log x<L'$. By Theorem~\ref{thm:main},
$L=(r-1)q_rk(\log k+\frac{r}{r-1}\log\log k+O(1))$ and
$L'-L=(r-1)q_r\log k+O(\log\log k)+o(k)=o(k)$.
In particular $L\sim\log x$, so $k\asymp\log x/\log\log x$, whence
$L=\log x+o(k)=(\log x)(1+o(1/\log\log x))$ and
$\lambda:=\log L=\log\log x+o(1/\log\log x)$.
On the other hand, $\lambda=\log k+\log\log k+\log((r-1)q_r)+O(\log\log k/\log k)$, so
$\log k=\lambda-\log\log k-\log((r-1)q_r)+O(\log\log k/\log k)$ and hence
\[
  L=(r-1)q_r\,k\Bigl(\lambda+\tfrac1{r-1}\log\log k+O(1)\Bigr),\qquad
  k=\frac{\zeta(r)}{r-1}\cdot\frac{L}{\lambda}
  \Bigl(1+\frac{\frac1{r-1}\log\log k+O(1)}{\lambda}\Bigr)^{-1}.
\]
Since $L/\lambda=\frac{\log x}{\log\log x}(1+o(1/\log\log x))$ and
$\log\log k=\log\lambda+o(1)=\log\log\log x+o(1)$, the assertion follows.
\end{proof}

\section{The constructive threshold: \texorpdfstring{$\pi^{2}/12$ versus $\pi^{2}/6$}{pi\^{}2/12 versus pi\^{}2/6}}
\label{sec:constructive}

By \eqref{eq:grimmett}, and more precisely by Corollary~\ref{cor:threshold}, the
first-moment constant is $\zeta(r)/(r-1)$. We compare this with what the classical
Chinese-remainder construction gives; the two differ by a factor $r/(r-1)$, that is, by a
factor $2$ when $r=2$. For $r=2$ this is the discrepancy between the
constant $\pi^{2}/12$ in \eqref{eq:erdos-lower} and the constant $\pi^{2}/6$ printed in
\cite[(20)]{Erdos1951}. Grimmett \cite[Appendix~1]{Grimmett} already observed that the
construction gives $\pi^{2}/12$, and that permuting the roles of the primes produces at least
$\exp((1-\varepsilon)q_2k\log k)$ disjoint runs of length $k$ below
$\exp(2(1+\varepsilon)q_2k\log k)$, so that $\pi^{2}/6$ would follow if these runs were
uniformly spread. Proposition~\ref{prop:crt} below complements this observation: no
covering system, whatever primes and residues it uses, has modulus smaller than
$\exp((rq_r+o(1))k\log k)$, so the standard Chinese-remainder bound $n_0<M$ for the least
non-negative solution $n_0$ of a single covering system cannot yield a constant larger than
$\pi^{2}/12$.

\begin{definition*}
A \emph{covering system of length $k$} is a finite set $P$ of primes together with residues
$\rho_p\bmod p^{r}$, $p\in P$, such that for every $i\in[1,k]$ there is a $p\in P$ with
$\rho_p\equiv-i\ (\mathrm{mod}\ p^{r})$. Its \emph{modulus} is $M=\prod_{p\in P}p^{r}$.
\end{definition*}

Any $n$ with $n\equiv\rho_p\ (p^{r})$ for all $p\in P$ lies in $D^{(r)}_k$, and by the
Chinese remainder theorem the least non-negative such $n$ is smaller than $M$. Using
$n_0<M$ as the upper bound on the location of the run is the device behind
\eqref{eq:erdos-lower}, and each event $A_\varphi$ in \S\ref{sec:lower} is of this form.

\begin{proposition}[the constructive threshold]\label{prop:crt}
Every covering system of length $k$ has modulus
\[
  M\ \ge\ \exp\bigl((r\,q_r+o(1))\,k\log k\bigr).
\]
Consequently, bounding the least positive solution of a covering system solely by its
modulus $M$ guarantees a run of length $k$ only below
$\exp((rq_r+o(1))k\log k)$, and thus proves no more than: for infinitely many
$i$,
\[
  s_{i+1}-s_i>(1+o(1))\,\frac{\zeta(r)}{r}\cdot\frac{\log s_i}{\log\log s_i},
\]
which for $r=2$ is the constant $\pi^{2}/12$. The bound on $M$ is attained by the system
$A_\varphi$ of \S\ref{sec:lower} in which $\varphi(V)$ consists of the $|V|$ smallest primes
exceeding $k^{1/r}$.
\end{proposition}

\begin{proof}
Let $z$ be as in \S\ref{sec:upper}, so that $\prod_{p\le z}p^{r}\le\sqrt k$ and
$z\asymp\log k$, and let $n$ be
a solution of the system. By Lemma~\ref{lem:sieve}, at least $q_rk(1+o(1))$ of the indices
$i\le k$ satisfy $p^{r}\nmid n+i$ for all $p\le z$; by the computation in Step 2 of
\S\ref{sec:upper}, all but $O(k/(\log^{r-1}k\log\log k))$ of those additionally satisfy
$p^{r}\nmid n+i$ for all $p\le k^{1/r}$. Each such $i$ is covered by some $p\in P$ with
$p>k^{1/r}$, and distinct such $i$ are covered by distinct primes, since $p^{r}>k>|i-i'|$
prevents $p^{r}\mid n+i$ and $p^{r}\mid n+i'$ simultaneously. Hence $P$ contains at least
$m=q_rk(1+o(1))$ primes exceeding $k^{1/r}$, and
\[
  \log M\ \ge\ r\bigl(\theta(P_{\pi(k^{1/r})+m})-\theta(k^{1/r})\bigr)
  =r\,m\log m\,(1+o(1))=r\,q_r\,k\log k\,(1+o(1))
\]
by the prime number theorem. For the second statement put
$\log x=\log M$ and invert as in Corollary~\ref{cor:threshold}; the exponent is larger by
the factor $r/(r-1)$, so the resulting constant is smaller by that factor.
\end{proof}

\begin{remark}\label{rem:status}
Comparing Proposition~\ref{prop:crt} with Corollary~\ref{cor:threshold}, for $r=2$ the two
thresholds are
\[
  \underbrace{\frac{\pi^{2}}{12}\cdot\frac{\log x}{\log\log x}}_{\text{covering systems}}
  \qquad\text{and}\qquad
  \underbrace{\frac{\pi^{2}}{6}\cdot\frac{\log x}{\log\log x}}_{\text{first moment}} .
\]
They differ because $\Delta^{(r)}_k$ exceeds the density $1/M$ of a single covering system
by the factor $m!\,e_m\prod_{j\le m}P_{\pi(k^{1/r})+j}^{r}=\exp((q_r+o(1))k\log k)$, which
comes from the $m!$ assignments of primes to positions. The argument of \S\ref{sec:lower}
keeps this factor, which is why the Chung--Erd\H{o}s step there cannot be replaced by a
single congruence class. Consider the inequality with the constant as printed in
\cite[(20)]{Erdos1951}: for infinitely many $i$,
\begin{equation}\label{eq:erdos-strong}
  s_{i+1}-s_i>(1+o(1))\,\frac{\pi^{2}}{6}\,\frac{\log s_i}{\log\log s_i}.
\end{equation}
Write $F(k)$ for the first occurrence of a run of $k$ consecutive non-squarefree integers.
By \eqref{eq:grimmett}, \eqref{eq:erdos-strong} is equivalent to the assertion that
$F(k)\le\Delta_k^{-1+o(1)}$ for infinitely many $k$, that is, that runs of length $k$ occur
near the first-moment threshold and not merely near the minimal covering-system modulus
$M_{\min}(k)=\Delta_k^{-2+o(1)}$. Note that every run of length $k$ starting at $n_0$ is
realised by a covering system of modulus $M$ (whose least non-negative solution is $\le n_0$,
and equals $n_0$ when $n_0=F(k)$), and $n_0$ can be much smaller than $M$ (as pointed out in
\cite[Appendix~B]{Cellarosi}, a lower bound on $M$ together with the Chinese-remainder
inequality $n_0<M$ gives no lower bound on $n_0$ itself); what Proposition~\ref{prop:crt}
shows is that any proof of a lower bound on gaps that bounds $n_0$ only by the modulus $M$
of a covering system cannot go beyond $\pi^{2}/12$, whereas establishing
\eqref{eq:erdos-strong} requires showing that for infinitely many $k$ there is a covering
system of length $k$ whose least solution satisfies $n_0\le M_{\min}(k)^{1/2+o(1)}$ (in
particular $n_0\le M^{1/2+o(1)}$).
We do not know a proof of \eqref{eq:erdos-strong}.
\end{remark}

\begin{question}\label{q:erdos20}
Does \eqref{eq:erdos-strong} hold? Equivalently, is $F(k)\le\Delta_k^{-1+o(1)}$ for
infinitely many $k$?
\end{question}

For the eighteen known values of $F(k)$ the ratio $F(k)\Delta_k$ lies in $[0.17,3.67]$
(Table~\ref{tab:delta} and Figure~\ref{fig:pred}), while the covering-system modulus is
$\Delta_k^{-2+o(1)}$. By Corollary~\ref{cor:threshold}, no first-moment argument can give a
constant larger than $\pi^{2}/6$ in \eqref{eq:erdos-strong}; this makes precise Erd\H{o}s's
remark in \cite{Erdos1951} that it ``seems to be extremely hard to replace $\pi^{2}/6$ by
any larger constant''.

\section{Numerical results}
\label{sec:numerics}

\subsection{Exact values}
We compute $\Delta^{(r)}_k$ by the usual inclusion--exclusion of Mirsky type. Let $Z$ be a finite set of primes containing every $p$ with
$p^{r}\le k$, and $M=\prod_{p\in Z}p^{r}$; for $r=2$ and $k\le168$ one may take
$Z=\{2,3,5,7,11\}$,
$M=5\,336\,100$, and for $r=3$ and $k\le124$ one may take $Z=\{2,3\}$, $M=216$. Every prime
$p\notin Z$ then has $p^{r}>k$ and so covers at most one position of a window of length
$k$, with probability $p^{-r}$, independently of $n\bmod M$. Consequently, if $u(\rho)$
denotes the number of positions of $[1,k]$ not covered by $p^{r}$, $p\in Z$, when
$n\equiv\rho\ (M)$, then
\begin{equation}\label{eq:exact}
  \Delta^{(r)}_k=\frac1M\sum_{\rho\bmod M}\mathcal Q\bigl(u(\rho)\bigr),\qquad
  \mathcal Q(u)=\sum_{j=0}^{u}(-1)^{j}\binom{u}{j}G(j),\qquad
  G(j)=\prod_{p\notin Z}\Bigl(1-\frac j{p^{r}}\Bigr).
\end{equation}
The histogram of $u(\rho)$ is computed by one pass of a sliding window
over $\Z/M\Z$; the products $G(j)$ are evaluated from prime zeta values, and the
alternating sum defining $\mathcal Q(u)$ is carried out in $900$-digit arithmetic (the
cancellation is severe: for $r=2$, $k=150$ the terms have size $\approx10^{40}$ and the result is
$\approx10^{-239}$). The same method with $G(j)$ replaced by $G(j+2)$ evaluates
$\alpha_r(h)$, since requiring the two endpoints to be uncovered as well simply adds two
positions to the inclusion--exclusion.

\subsection{Checks}
(i) For $r=2$, $\Delta^{(2)}_1$ agrees with $1-6/\pi^{2}=0.39207289814597337134\ldots$ to
$20$ digits; for $r=3$, $\Delta^{(3)}_1$ agrees with $1-1/\zeta(3)=0.168092627419\ldots$
to $14$ digits. (ii) $\sum_h\alpha(h)=6/\pi^{2}$ and $\sum_h h\,\alpha(h)=1$ hold to $20$
digits, the latter being the statement that the gaps tile $\N$. (iii)
$\alpha(1)/\sum_h\alpha(h)=0.53071182\ldots$ and $\alpha(2)/\sum_h\alpha(h)=0.324294\ldots$
agree with the constants recorded in \cite[\texttt{A076259}]{OEIS}. (iv) A segmented sieve
over $[1,10^{12}]$, using a wheel modulo $2^{2}3^{2}5^{2}7^{2}$ and marking only the squares
of primes $\ge11$, gives the counts of Table~\ref{tab:sieve}; the longest run below
$10^{12}$ is $13$, consistent with $F(14)=1.0435\cdot10^{12}$. (v) The identity
$\alpha(k+1)=\Delta_k-2\Delta_{k+1}+\Delta_{k+2}$ of \eqref{eq:D-alpha} holds to the
printed precision for all $k\le148$.

\begin{table}[ht]
\caption{Sieve check for $r=2$. $N_k=\#(D^{(2)}_k\cap[1,10^{12}])$ against
$\Delta^{(2)}_k\cdot10^{12}$.}
\label{tab:sieve}
\begin{tabular}{rrrl}
\toprule
$k$ & $N_k$ (sieved) & $\Delta^{(2)}_k\cdot10^{12}$ & ratio\\
\midrule
1 & $392\,072\,897\,726$ & $392\,072\,898\,146$ & $1.0000000$\\
3 & $18\,634\,011\,558$ & $18\,634\,010\,350$ & $1.0000001$\\
5 & $641\,333\,605$ & $641\,332\,384$ & $1.0000019$\\
7 & $5\,206\,950$ & $5\,206\,270$ & $1.000131$\\
9 & $73\,177$ & $73\,077$ & $1.00137$\\
11 & $725$ & $754$ & $0.961$\\
13 & $8$ & $15.5$ & (Poisson)\\
\bottomrule
\end{tabular}
\end{table}

\subsection{The constant $C_r$}\label{sub:Cr}
Theorem~\ref{thm:main} asserts that
\[
  R_r(k):=\frac1k\Bigl(\log\frac1{\Delta^{(r)}_k}-(r-1)q_r\,k\log k
  -r\,q_r\,k\log\log k\Bigr)\longrightarrow -C_r .
\]
Figure~\ref{fig:Cr} plots $R_r(k)+C_r$ for $r=2$, $10\le k\le150$ and for $r=3$,
$10\le k\le120$. Both curves descend towards $0$, from $0.96$ and $0.94$ respectively to
$0.11$ and $0.09$. Here $C_2=1.45955\ldots$ and $C_3=2.44410\ldots$ are computed from
\eqref{eq:Cr}. The size of the residual is consistent with the $O(m\log\log m/\log m)$ error carried by \eqref{eq:logs}, which at $k=150$
amounts to about $0.39$ per unit of $k$.

\begin{figure}[ht]
\centering
\begin{tikzpicture}[x=0.0755cm,y=4.6cm]
  \draw[gray!45] (10,0)--(152,0);
  \foreach \y in {0.2,0.4,0.6,0.8,1.0}{\draw[gray!25](10,\y)--(152,\y);
    \node[left,font=\scriptsize,gray!70] at (10,\y){$\y$};}
  \node[left,font=\scriptsize,gray!70] at (10,0){$0$};
  \draw (10,0)--(10,1.05);
  \foreach \x in {25,50,75,100,125,150}{\draw(\x,0)--(\x,-0.022);
    \node[below,font=\scriptsize] at (\x,-0.022){$\x$};}
  \node[below,font=\scriptsize] at (81,-0.085){$k$};
  \draw[thick] plot coordinates{(10,0.9589) (15,0.7051) (20,0.5921) (25,0.4662) (30,0.4168) (35,0.4393) (40,0.3860) (45,0.3441) (50,0.3085) (55,0.2736) (60,0.2279) (65,0.1906) (70,0.2213) (75,0.2063) (80,0.2404) (85,0.1700) (90,0.2090) (95,0.2010) (100,0.2358) (105,0.1573) (110,0.1378) (115,0.1600) (120,0.1504) (125,0.1219) (130,0.1095) (135,0.1080) (140,0.1299) (145,0.1206) (150,0.1131)};
  \foreach \p in {(10,0.9589),(15,0.7051),(20,0.5921),(25,0.4662),(30,0.4168),(35,0.4393),(40,0.3860),(45,0.3441),(50,0.3085),(55,0.2736),(60,0.2279),(65,0.1906),(70,0.2213),(75,0.2063),(80,0.2404),(85,0.1700),(90,0.2090),(95,0.2010),(100,0.2358),(105,0.1573),(110,0.1378),(115,0.1600),(120,0.1504),(125,0.1219),(130,0.1095),(135,0.1080),(140,0.1299),(145,0.1206),(150,0.1131)}{\fill \p circle(1.2pt);}
  \draw[thick,dashed] plot coordinates{(10,0.9370) (15,0.9368) (20,0.6303) (25,0.4663) (30,0.3283) (35,0.1976) (40,0.3449) (45,0.2565) (50,0.1938) (55,0.1679) (60,0.0888) (65,0.0533) (70,0.1548) (75,0.1152) (80,0.2196) (85,0.0569) (90,0.0232) (95,0.1130) (100,0.0834) (105,0.0640) (110,0.0460) (115,0.0138) (120,0.0925)};
  \foreach \p in {(10,0.9370),(15,0.9368),(20,0.6303),(25,0.4663),(30,0.3283),(35,0.1976),(40,0.3449),(45,0.2565),(50,0.1938),(55,0.1679),(60,0.0888),(65,0.0533),(70,0.1548),(75,0.1152),(80,0.2196),(85,0.0569),(90,0.0232),(95,0.1130),(100,0.0834),(105,0.0640),(110,0.0460),(115,0.0138),(120,0.0925)}{\draw[fill=white,line width=0.6pt] \p circle(1.5pt);}
  \draw[thick] (95,0.93)--(105,0.93); \fill (100,0.93) circle(1.2pt);
  \node[right,font=\scriptsize] at (106,0.93){$r=2$ (squarefree), $C_2=1.45955$};
  \draw[thick,dashed] (95,0.83)--(105,0.83);
  \draw[fill=white,line width=0.6pt] (100,0.83) circle(1.5pt);
  \node[right,font=\scriptsize] at (106,0.83){$r=3$ (cubefree), $C_3=2.44410$};
\end{tikzpicture}
\caption{$R_r(k)+C_r$ for $r=2$ and $r=3$, with $C_r$ given by the closed form
\eqref{eq:Cr}. By Theorem~\ref{thm:main} both tend to $0$. The oscillation is
arithmetic in origin, as in Figure~\ref{fig:pred}.}
\label{fig:Cr}
\end{figure}

\subsection{The coefficient $r\,q_r$}\label{sub:coeff}
A direct fit of $\Delta_k$ for $k\le150$ does not show the $\log\log$ term (see
Remark~\ref{rem:slow}). It becomes visible in the factor $\mathcal Q(u)$ of
\eqref{eq:exact}, which carries the second-order term. For $r=2$ the saddle point analysis of \S\ref{sec:upper} predicts,
with $s=s(u)$ defined by \eqref{eq:saddle},
\begin{equation}\label{eq:Qderiv}
  \frac{d}{du}\log\frac1{\mathcal Q(u)}=\log u+2\log\log s(u)+2\log\tfrac2\pi+o(1),
  \qquad 2\log\tfrac2\pi=-0.9032\ldots
\end{equation}
Figure~\ref{fig:coeff} compares \eqref{eq:Qderiv} with the exact symmetric difference
quotient of $\log(1/\mathcal Q(u))$ computed from the $900$-digit values of $\mathcal Q$.
The difference decreases monotonically from $0.374$ at $u=30$ to $0.019$ at $u=145$, by a
factor of about $20$, consistent with the coefficient $2$.

\begin{figure}[ht]
\centering
\begin{tikzpicture}[x=0.0785cm,y=3.05cm]
  \foreach \y in {1.4,1.8,2.2,2.6}{\draw[gray!25](10,\y)--(148,\y);
    \node[left,font=\scriptsize,gray!70] at (10,\y){$\y$};}
  \draw (10,1.15)--(148,1.15); \draw (10,1.15)--(10,2.95);
  \foreach \x in {25,50,75,100,125}{\draw(\x,1.15)--(\x,1.12);
    \node[below,font=\scriptsize] at (\x,1.12){$\x$};}
  \node[below,font=\scriptsize] at (79,1.005){$u$};
  \draw[thick,dashed] plot coordinates{(10,1.2427) (15,1.6097) (20,1.8223) (25,1.9685) (30,2.0783) (35,2.1654) (40,2.2371) (45,2.2978) (50,2.3502) (55,2.3962) (60,2.4371) (65,2.4738) (70,2.5070) (75,2.5374) (80,2.5653) (85,2.5911) (90,2.6150) (95,2.6374) (100,2.6583) (105,2.6780) (110,2.6965) (115,2.7140) (120,2.7306) (125,2.7464) (130,2.7614) (135,2.7757) (140,2.7894) (145,2.8025)};
  \draw[thick] plot coordinates{(10,2.4844) (15,2.4260) (20,2.4196) (25,2.4322) (30,2.4522) (35,2.4751) (40,2.4985) (45,2.5217) (50,2.5441) (55,2.5655) (60,2.5859) (65,2.6053) (70,2.6238) (75,2.6414) (80,2.6581) (85,2.6741) (90,2.6893) (95,2.7038) (100,2.7178) (105,2.7311) (110,2.7439) (115,2.7562) (120,2.7681) (125,2.7795) (130,2.7905) (135,2.8011) (140,2.8114) (145,2.8213)};
  \foreach \p in {(10,2.4844),(15,2.4260),(20,2.4196),(25,2.4322),(30,2.4522),(35,2.4751),(40,2.4985),(45,2.5217),(50,2.5441),(55,2.5655),(60,2.5859),(65,2.6053),(70,2.6238),(75,2.6414),(80,2.6581),(85,2.6741),(90,2.6893),(95,2.7038),(100,2.7178),(105,2.7311),(110,2.7439),(115,2.7562),(120,2.7681),(125,2.7795),(130,2.7905),(135,2.8011),(140,2.8114),(145,2.8213)}{\fill \p circle(1.3pt);}
  \draw[thick] (56,1.45)--(66,1.45); \fill (61,1.45) circle(1.3pt);
  \node[right,font=\scriptsize] at (67,1.45){exact $\frac{d}{du}\log(1/\mathcal Q(u))-\log u$};
  \draw[thick,dashed] (56,1.32)--(66,1.32);
  \node[right,font=\scriptsize] at (67,1.32){saddle point $2\log\log s(u)-0.9032$};
\end{tikzpicture}
\caption{The coefficient $r\,q_r$ for $r=2$: the exact difference quotient against the
prediction \eqref{eq:Qderiv}. The difference falls from $0.374$ at $u=30$ to $0.019$ at
$u=145$.}
\label{fig:coeff}
\end{figure}

Solving the saddle equation with the actual primes (exactly for $p\le2\cdot10^{7}$, by the
prime number theorem integral beyond) allows \eqref{eq:Qderiv} to be extrapolated far
beyond the range in which $\mathcal Q(u)$ itself is computable. Table~\ref{tab:extrap} shows
that the resulting ratio approaches $2$ very slowly, for large $u$ roughly like
$2-0.589/\log\log u$; at $u=10^{30}$ it is still $1.81$.

\begin{table}[ht]
\caption{$\bigl(\frac{d}{du}\log(1/\mathcal Q(u))-\log u\bigr)/\log\log u$ for $r=2$, from the
saddle point with the actual primes. The limit is $2$.}
\label{tab:extrap}
\begin{tabular}{lrrrrrrr}
\toprule
$u$ & $10^{2}$ & $10^{4}$ & $10^{6}$ & $10^{9}$ & $10^{12}$ & $10^{18}$ & $10^{30}$\\
\midrule
ratio & $1.776$ & $1.749$ & $1.767$ & $1.780$ & $1.788$ & $1.799$ & $1.812$\\
$2-0.589/\log\log u$ & $1.614$ & $1.735$ & $1.776$ & $1.806$ & $1.823$ & $1.842$ & $1.861$\\
\bottomrule
\end{tabular}
\end{table}

\subsection{Values for $r=2$}
Table~\ref{tab:delta} lists $\Delta^{(2)}_k$ and $\alpha(k+1)$ for $k\le20$ together with
the first occurrence $F(k)$ of a run of length $k$ (\texttt{A045882}) and the ratio
$F(k)\Delta_k$. For $k\le18$ the ratio lies in $[0.17,3.67]$, with geometric mean $0.908$,
so $1/\Delta_k$ gives $F(k)$ to within a factor $6$. Figure~\ref{fig:pred} shows this
together with the estimates of \S\ref{sub:pred}.

\begin{table}[ht]
\caption{Exact densities for $r=2$, and the first occurrence of a run of length $k$.}
\label{tab:delta}
\begin{tabular}{rllrr}
\toprule
$k$ & $\Delta^{(2)}_k$ & $\alpha(k+1)$ & $F(k)$ & $F(k)\Delta_k$\\
\midrule
1 & $0.392073$ & $0.197147$ & $4$ & $1.568$ \\
2 & $0.10678$ & $0.0716601$ & $8$ & $0.854$ \\
3 & $0.018634$ & $0.0149788$ & $48$ & $0.894$ \\
4 & $0.00214826$ & $0.000938535$ & $242$ & $0.520$ \\
5 & $0.000641332$ & $0.000500664$ & $844$ & $0.541$ \\
6 & $7.29375\cdot10^{-5}$ & $6.27919\cdot10^{-5}$ & $22020$ & $1.606$ \\
7 & $5.20627\cdot10^{-6}$ & $4.74547\cdot10^{-6}$ & $217070$ & $1.130$ \\
8 & $2.66941\cdot10^{-7}$ & $1.25699\cdot10^{-7}$ & $1092747$ & $0.292$ \\
9 & $7.30768\cdot10^{-8}$ & $6.40077\cdot10^{-8}$ & $8870024$ & $0.648$ \\
10 & $4.91166\cdot10^{-9}$ & $3.44829\cdot10^{-9}$ & $221167422$ & $1.086$ \\
11 & $7.54184\cdot10^{-10}$ & $6.79634\cdot10^{-10}$ & $221167422$ & $0.167$ \\
12 & $4.50046\cdot10^{-11}$ & $1.46822\cdot10^{-11}$ & $47255689915$ & $2.127$ \\
13 & $1.54595\cdot10^{-11}$ & $1.42864\cdot10^{-11}$ & $82462576220$ & $1.275$ \\
14 & $5.96503\cdot10^{-13}$ & $5.57143\cdot10^{-13}$ & $1043460553364$ & $0.622$ \\
15 & $1.99409\cdot10^{-14}$ & $1.90509\cdot10^{-14}$ & $79180770078548$ & $1.579$ \\
16 & $5.22193\cdot10^{-16}$ & $2.17484\cdot10^{-16}$ & $3215226335143218$ & $1.679$ \\
17 & $1.54404\cdot10^{-16}$ & $1.46306\cdot10^{-16}$ & $23742453640900972$ & $3.666$ \\
18 & $4.09765\cdot10^{-18}$ & $3.91514\cdot10^{-18}$ & $125781000834058568$ & $0.515$ \\
19 & $9.79825\cdot10^{-20}$ & $7.75381\cdot10^{-20}$ & --- & --- \\
20 & $1.34579\cdot10^{-20}$ & $6.56585\cdot10^{-22}$ & --- & --- \\
\bottomrule
\end{tabular}
\end{table}

\begin{figure}[ht]
\centering
\begin{tikzpicture}[x=0.335cm,y=0.152cm]
  \draw[gray!35] (1,0)--(32,0);
  \foreach \y in {5,10,15,20,25,30,35}{\draw[gray!25](1,\y)--(32,\y);
    \node[left,font=\scriptsize,gray!70] at (1,\y){$10^{\y}$};}
  \draw (1,0)--(32,0); \draw (1,0)--(1,37);
  \foreach \x in {5,10,15,20,25,30}{\draw(\x,0)--(\x,-0.7);
    \node[below,font=\scriptsize] at (\x,-0.7){$\x$};}
  \node[below,font=\scriptsize] at (16.5,-3.4){$k$};
  \draw[dashed,gray] (18.5,0)--(18.5,36);
  \node[right,font=\scriptsize,gray!80,rotate=90] at (19.1,3){exhaustive search};
  \draw[thick] plot coordinates{(1,0.407) (2,0.972) (3,1.730) (4,2.668) (5,3.193) (6,4.137) (7,5.283) (8,6.574) (9,7.136) (10,8.309) (11,9.123) (12,10.347) (13,10.811) (14,12.224) (15,13.700) (16,15.282) (17,15.811) (18,17.387) (19,19.009) (20,19.871) (21,20.189) (22,21.849) (23,23.672) (24,25.420) (25,25.893) (26,27.657) (27,29.035) (28,30.562) (29,30.973) (30,32.745) (31,34.542) (32,36.413)};
  \foreach \p in {(1,0.407),(2,0.972),(3,1.730),(4,2.668),(5,3.193),(6,4.137),(7,5.283),(8,6.574),(9,7.136),(10,8.309),(11,9.123),(12,10.347),(13,10.811),(14,12.224),(15,13.700),(16,15.282),(17,15.811),(18,17.387),(19,19.009),(20,19.871),(21,20.189),(22,21.849),(23,23.672),(24,25.420),(25,25.893),(26,27.657),(27,29.035),(28,30.562),(29,30.973),(30,32.745),(31,34.542),(32,36.413)}{\fill \p circle(1.4pt);}
  \foreach \p in {(1,0.602),(2,0.903),(3,1.681),(4,2.384),(5,2.926),(6,4.343),(7,5.337),(8,6.039),(9,6.948),(10,8.345),(11,8.345),(12,10.674),(13,10.916),(14,12.018),(15,13.899),(16,15.507),(17,16.376),(18,17.100)}{\draw[fill=white,line width=0.7pt] \p +(-2.3pt,-2.3pt) rectangle +(2.3pt,2.3pt);}
  \foreach \p in {(19,19.496),(20,23.171),(21,20.507),(22,23.330),(23,26.837),(24,31.456)}{\draw[line width=0.7pt] \p +(0pt,2.8pt) -- +(2.6pt,-2.1pt) -- +(-2.6pt,-2.1pt) -- cycle;}
  \draw[thick] (2.2,36)--(4.2,36); \fill (3.2,36) circle(1.4pt);
  \node[right,font=\scriptsize] at (4.4,36){$1/\Delta_k$ (model)};
  \draw[fill=white,line width=0.7pt] (3.2,33.4) +(-2.3pt,-2.3pt) rectangle +(2.3pt,2.3pt);
  \node[right,font=\scriptsize] at (4.4,33.4){$F(k)$ (known)};
  \draw[line width=0.7pt] (3.2,30.8) +(0pt,2.8pt)--+(2.6pt,-2.1pt)--+(-2.6pt,-2.1pt)--cycle;
  \node[right,font=\scriptsize] at (4.4,30.8){Wong upper bound};
\end{tikzpicture}
\caption{The model $1/\Delta^{(2)}_k$ against the known first occurrences $F(k)$ and
against E.~B.~Wong's Chinese-remainder upper bounds for $F(k)$, $19\le k\le24$, recorded in
\cite{Marmet}. The vertical axis is $\log_{10}$. The model agrees with all eighteen known
records to within a factor $6$ and stays below all six upper bounds. The zig-zag reflects the
residue structure modulo $4$ and $9$ (for instance
$\Delta_{20}/\Delta_{21}=2.1$ while $\Delta_{21}/\Delta_{22}=46$).}
\label{fig:pred}
\end{figure}

\subsection{The constants $B(\gamma)$}
The values of $\alpha(h)$ give the constants in \eqref{eq:moments}. Since $\alpha(h)$
decays like $\exp(-q_2h\log h)$ by \eqref{eq:grimmett}, truncating
$B(\gamma)=\sum_hh^{\gamma}\alpha(h)$ at $h=60$ suffices for the precision shown. We find
\[
  B(0)=\tfrac{6}{\pi^{2}},\quad B(1)=1,\quad
  B(2)=2.04070977646714\ldots,\quad B(3)=5.04286811382262\ldots,
\]
\[
  B(3.6875)=10.2852261048655\ldots,\qquad
  B(3.75)=11.0089932906174\ldots,
\]
\[
  B(4)=14.5232122094753\ldots
\]
($B(0)$ and $B(1)$ are forced and serve as checks; $B(2)$ and $B(3)$ are the constants in
the theorems of Erd\H{o}s and Hooley, and $3.6875$ and $3.75$ are the endpoints of the ranges
of Huxley \cite{Huxley2000} and Chan \cite{Chan}).

\subsection{Estimates for $F(k)$, $19\le k\le32$}\label{sub:pred}
Marmet \cite[Appendix A]{Marmet} gave estimates of this kind, based on an approximation to
the probability of assembling the minimum number of primes needed to produce a gap of a
given length, and compared them with the known values. We use the exact density $\Delta_k$
instead, with the factor $0.908$ from Table~\ref{tab:delta}. If runs of a given length
first occur at the scale given by their density, as they do for $k\le18$, then
$F(k)\approx0.908/\Delta_k$. Table~\ref{tab:pred} lists this for $19\le k\le32$; the
uncertainty is roughly a factor $4$ either way. Exhaustive search reached $1.2587\cdot10^{17}$ in \cite{Marmet} and
$10^{18}$ in \cite{MOT}. All the estimates lie below the Chinese-remainder upper bounds
of E.~B.~Wong recorded in \cite{Marmet}, and for $k=21$ the estimate is only a factor
$2.3$ below Wong's bound.

\begin{table}[ht]
\caption{Estimated first occurrence of a run of $k$ consecutive non-squarefree integers,
and Wong's upper bounds where available.}
\label{tab:pred}
\begin{tabular}{rlll}
\toprule
$k$ & $1/\Delta_k$ & estimate $0.908/\Delta_k$ & Wong upper bound\\
\midrule
19 & $1.021\cdot10^{19}$ & $9.27\cdot10^{18}$ & $3.13\cdot10^{19}$\\
20 & $7.431\cdot10^{19}$ & $6.75\cdot10^{19}$ & $1.48\cdot10^{23}$\\
21 & $1.545\cdot10^{20}$ & $1.40\cdot10^{20}$ & $3.21\cdot10^{20}$\\
22 & $7.063\cdot10^{21}$ & $6.41\cdot10^{21}$ & $2.14\cdot10^{23}$\\
23 & $4.695\cdot10^{23}$ & $4.26\cdot10^{23}$ & $6.87\cdot10^{26}$\\
24 & $2.627\cdot10^{25}$ & $2.39\cdot10^{25}$ & $2.85\cdot10^{31}$\\
25 & $7.814\cdot10^{25}$ & $7.09\cdot10^{25}$ & ---\\
26 & $4.540\cdot10^{27}$ & $4.12\cdot10^{27}$ & ---\\
27 & $1.085\cdot10^{29}$ & $9.85\cdot10^{28}$ & ---\\
28 & $3.647\cdot10^{30}$ & $3.31\cdot10^{30}$ & ---\\
29 & $9.399\cdot10^{30}$ & $8.53\cdot10^{30}$ & ---\\
30 & $5.563\cdot10^{32}$ & $5.05\cdot10^{32}$ & ---\\
31 & $3.480\cdot10^{34}$ & $3.16\cdot10^{34}$ & ---\\
32 & $2.589\cdot10^{36}$ & $2.35\cdot10^{36}$ & ---\\
\bottomrule
\end{tabular}
\end{table}

\begin{remark}\label{rem:slow}
A naive fit of $\log(1/\Delta^{(2)}_k)=q_2k(\log k+A\log\log k+C)$ to the exact values for
$k\le150$ returns $A\approx0.8$, not $2$. There are two reasons. First, $\log\log k$ ranges
only over $[1.2,1.6]$ for $30\le k\le150$, so the fit is badly conditioned and a change in
$A$ can be compensated by $C$. Second, in that range the sum \eqref{eq:exact} is dominated not by the
typical value of $u(\rho)$ but by its minimum: for $k=150$ the mean of $u(\rho)$ is
$93.26$ while the single term $u=89=\min_\rho u(\rho)$ carries $98.7\%$ of $\Delta_{150}$.
This happens because the window length is far below the period $M$; by
Lemma~\ref{lem:sieve} it does not happen for large $k$. For this reason the coefficient
and $C_r$ are tested in \S\ref{sub:coeff} and \S\ref{sub:Cr} by other means.
\end{remark}

\appendix

\section{A reformulation of the exponent in the moment problem}\label{app:moment}

This appendix is not used elsewhere in the paper. It isolates the step in the treatment of
\eqref{eq:moments} by Huxley \cite{Huxley1997} and Chan \cite{Chan} at which the exponent
$3.75$ is decided, and records how that step depends on the order of the derivative test
used in it; this locates the obstruction referred to in Remark~\ref{rem:nogain}. We take
the framework of \cite{Huxley1997} in the form in which it is summarised in
\cite[\S2]{Chan}; we have not seen \cite{Huxley1997} itself. For the order $5$ the
computation below is the proof of \cite[Theorem~1]{Chan}, with the order kept as a
parameter. Logarithmic factors are suppressed throughout.

\subsection*{The framework}
One splits the gap length $H$ and the prime scale $P$ dyadically, in the ranges
\[
  (\log x/\log\log x)^{1/(\gamma+2)}\ll H\ll x^{1/5}\log x,\qquad
  H\log H\ll P\ll H^{\gamma}\log H,
\]
and by \cite[Lemma~4]{Huxley1997} (see \cite[Lemma~1]{Chan}) it suffices to show, for each
pair $(H,P)$, either that $T(H,P)<H/(64\gamma\log H)$, where $T(H,P)$ is the largest
number, over $x/2\le n\le x-H$, of $i\in[1,H]$ for which $n+i$ is divisible by $p^{2}$ for
some prime $p\in[P,2P)$, or one of three bounds for a count $S(H,P)$ of sextuples
$(p_1,p_2,p_3,q_1,q_2,q_3)$ with $P\le p_j<2P$ and the $p_j^{2}q_j$ in a common window of
length $H$. The count $S(H,P)$ is split into three cases \cite[\S2]{Chan}. Case~1 is bounded
by $H^{6/5}x^{3/5}$, which is admissible for $\gamma<19/5$, the condition being critical at
$H=x^{1/5}$; for its subcase 1(b) this bound is already in \cite{Huxley1997}, and for
subcase 1(a) it is \cite[\S3]{Chan}, using \cite{Huxley2003}. Case~3 is admissible for
$\gamma<4$. Case~2 is admissible when $P\gg H^{(2\gamma-4)/3}$, and also when
$P\ll x/H^{2\gamma-3}$; these two ranges together contain every $P$ as soon as
$H\le x^{3/(8\gamma-13)}$.

\subsection*{The derivative test}
$T(H,P)$ is at most the maximum over $n$ of the number of integers $u\in[P,2P)$ with
$\|f_n(u)\|\le\delta$, where $f_n(u)=n/u^{2}$ and $\delta=H/P^{2}$; here
$|f_n^{(j)}(u)|\asymp\lambda_j:=x/P^{j+2}$. Two results of Huxley and Sargos bound this
number. For every order $\ell\ge2$, \cite[Th\'eor\`eme~1]{HS1995} gives
\[
  \ll P\lambda_\ell^{2/(\ell(\ell+1))}+P\delta^{2/(\ell(\ell-1))}+(\delta/\lambda_\ell)^{1/\ell}+1,
\]
and for $\ell\ge5$, \cite[Th\'eor\`eme~5]{HS2006} gives
the same bound with $\delta^{2/((\ell-1)(\ell-2))}$ in the second term and
$(\delta/\lambda_{\ell-1})^{1/(\ell-1)}$ in the third. For $\ell=3,4$ the same theorem has a
different form, discussed after \eqref{eq:caps}. Requiring the first two terms to be smaller
than $H$ gives the two thresholds $P\le H^{a}x^{-b}$ and $P\le H^{c}$, with
\[
  a=\frac{\ell(\ell+1)}{\ell^{2}-\ell-4},\qquad b=\frac{2}{\ell^{2}-\ell-4},\qquad
  c=\frac{1-e}{1-2e},
\]
where $e=2/(\ell(\ell-1))$ by \cite{HS1995} and $e=2/((\ell-1)(\ell-2))$ by \cite{HS2006} for
$\ell\ge5$.
For $P$ above a threshold the first alternative may fail, and $S(H,P)$ must then be handled
by Case~2. If $H\le x^{3/(8\gamma-13)}$ this is possible for every $P$; for larger $H$ it
requires $P\gg H^{(2\gamma-4)/3}$, that is $H^{a}x^{-b}\ge H^{(2\gamma-4)/3}$ and
$c\ge(2\gamma-4)/3$. This yields the two caps
\begin{equation}\label{eq:caps}
  \gamma_1(\ell)=\frac{3a+13b+4}{8b+2}=\frac{7\ell^{2}-\ell+10}{2(\ell^{2}-\ell+4)},
  \qquad \gamma_2(\ell)=2+\tfrac32 c .
\end{equation}
The third term, together with $H\ll x^{1/5}$, gives the cap $2+\frac32(\ell+4)/\ell$ by
\cite{HS1995} and $2+\frac32(\ell+3)/(\ell-1)$ by \cite{HS2006}; these equal $5$ at $\ell=4$ and at
$\ell=5$ respectively, and this cap is never the binding one in Table~\ref{tab:gamma}.

For $\ell=3$ and $\ell=4$, \cite[Th\'eor\`eme~5]{HS2006} gives
\begin{gather*}
  \ll P\lambda_3^{1/6}+P\delta^{1/4}\lambda_3^{1/12}+P\delta^{2/3}+(\delta/\lambda_2)^{1/2}+1,\\
  \ll P\lambda_4^{1/10}+P\delta^{1/7}\lambda_4^{1/21}+P\delta^{1/3}+(\delta/\lambda_3)^{1/3}+1
\end{gather*}
respectively. The first term is the same as before. The term $P\delta^{2/3}$ imposes no
condition, and $P\delta^{1/3}$ gives the cap $5$. The binding term is the mixed one: its
threshold is $P\le H^{9}x^{-1}$ for $\ell=3$ and $P\le H^{2}x^{-1/9}$ for $\ell=4$, and the
computation that gives $\gamma_1$ turns these into the caps $22/5$ and $103/26$. These are
the entries for $\ell=3,4$ in the fourth column of Table~\ref{tab:gamma}. Th\'eor\`eme~6 of
\cite{HS2006}, which holds for $\ell\ge4$ at the cost of a factor $P^{\varepsilon}$, gives no
better entry for any $\ell$ in the table.

\begin{table}[ht]
\caption{The caps \eqref{eq:caps} on the moment exponent as a function of the order $\ell$ of
the derivative test, with $\gamma_2$ computed from \cite{HS1995} and from
\cite[Th\'eor\`eme~5]{HS2006}; for $\ell=3,4$ the entry from \cite{HS2006} is the cap from its
mixed term. The last column is the smaller of $\gamma_1$ and the better $\gamma_2$.}
\label{tab:gamma}
\begin{tabular}{rlllll}
\toprule
$\ell$ & $\gamma_1(\ell)$ & $\gamma_2(\ell)$ by \cite{HS1995} & $\gamma_2(\ell)$ by \cite{HS2006}
  & $\min$ & \\
\midrule
$3$ & $7/2=3.5000$ & $5$ & $22/5=4.4000$ & $3.5000$ &\\
$4$ & $59/16=3.6875$ & $31/8=3.8750$ & $103/26=3.9615$ & $3.6875$ & cf.\ \cite{Huxley2000}\\
$5$ & $15/4=3.7500$ & $59/16=3.6875$ & $31/8=3.8750$ & $\mathbf{3.7500}$
  & \cite{Chan}\\
$6$ & $64/17=3.7647$ & $47/13=3.6154$ & $59/16=3.6875$ & $3.6875$ &\\
$7$ & $173/46=3.7609$ & $68/19=3.5789$ & $47/13=3.6154$ & $3.6154$ &\\
$8$ & $15/4=3.7500$ & $185/52=3.5577$ & $68/19=3.5789$ & $3.5789$ &\\
$10$ & $175/47=3.7234$ & $152/43=3.5349$ & $241/68=3.5441$ & $3.5441$ &\\
\bottomrule
\end{tabular}
\end{table}

\subsection*{Comparison with the literature}
For $\ell=5$ with \cite{HS2006} the three thresholds are $P\le H^{15/8}x^{-1/8}$, $P\le H^{5/4}$
and $P\le H^{3/4}x^{1/4}$, and the first leads to the condition $H>x^{3/(77-16\gamma)}$;
these coincide with the thresholds of \cite[\S4]{Chan}, and $\gamma_1(5)=15/4$,
$\gamma_2(5)=31/8$ and the cap $5$ from the third term are the three ranges
$\gamma<3.75$, $\gamma<3.875$, $\gamma<5$ that appear there.

The case $\ell=4$ is less direct. According to \cite[\S\S1,3,4]{Chan}, \cite{Huxley1997} used
the fourth-derivative case of \cite{HS1995}, and its range $\gamma<11/3$ was set by
Case~1(a), not by $T(H,P)$; \cite{Huxley2000}
obtained $\gamma<59/16$ by treating Case~1(a) with its Theorem~1. Since $11/3<59/16$, this is
consistent with the table: once Case~1(a) was improved, the binding constraint became the
fourth-derivative cap $\gamma_1(4)=59/16$. \cite{Huxley2000} itself only states the range,
in \S1, without the computation, so this account of where $59/16$ comes from is inferred and
has not been checked against \cite{Huxley1997}.

The table also shows why the gain to $15/4$ needed both ingredients of \cite{Chan}. With
\cite{HS1995} alone, $\ell=5$ gives no improvement over $\ell=4$, because $\gamma_2(5)$ is then
$59/16$; it is the exponent of $\delta$ in \cite{HS2006} that raises $\gamma_2(5)$ to $31/8$,
while the improvement of Case~1(a) is what lifts the ceiling to $19/5$. Higher orders do
not help: $\gamma_1$ is largest at $\ell=6$, but there $\gamma_2$ has already fallen back to
$59/16$.

\subsection*{The critical configuration}
At $\gamma=15/4$ and $\ell=5$ three constraints meet at $H=x^{3/17}$, $P=x^{7/34}$: the
first term of the derivative test, the lower range $P\gg H^{(2\gamma-4)/3}$ of Case~2, and
the condition $H\le x^{3/(8\gamma-13)}$ under which the two ranges of Case~2 overlap.
Improving any one of the three at this configuration enlarges the range of $\gamma$; Chan
suggests that further progress should come from Case~2 \cite[\S1]{Chan}. In the
parametrisation of the first term by its exponent $\kappa$, so that $a=1/(1-(\ell+2)\kappa)$
and $b=\kappa a$, the cap at $\ell=5$ becomes
\[
  \gamma_1=\frac{7-15\kappa}{2-6\kappa},
\]
which is increasing in $\kappa$ and equals $15/4$ at the available value
$\kappa=2/(\ell(\ell+1))=1/15$. Thus every $\kappa>1/15$ already gives a range beyond $3.75$;
$\kappa=13/189=0.0688\ldots$, an increase of $3.2\%$, gives $\gamma<3.76$, with critical
configuration $H=x^{0.17564}$, $P=x^{0.20609}$; and $\kappa=1/13$, an increase of $15.4\%$,
reaches the ceiling $19/5$ of Case~1. (For general $\ell$, $\gamma_1(\ell)\ge19/5$ is equivalent to
$a\ge5.8\,b+1.2$, that is $\kappa\ge0.2/(1.2\ell-3.4)$.) Such a gain is needed only for the
function $n/u^{2}$ in this range, not for the derivative test in general. Pandey's method
\cite{Pandey} bounds counts of the same type, namely the number of integers in a short
interval divisible by a large square, but in the critical ranges for $H\approx x^{1/5}$; we
do not know whether it applies to the configuration above.

\section*{Acknowledgements}
The computations were carried out in Python (\texttt{mpmath}) and C; the code and the
tables of $\Delta^{(r)}_k$ and $\alpha_r(h)$ are available from the author.
The author used Claude (Anthropic) to explore arguments, to write and check code, to locate
literature and to edit the text; the author checked all results and takes full
responsibility for the paper.
The author is grateful to the referee for drawing attention to the paper of Grimmett
\cite{Grimmett} and to the constant $\pi^{2}/12$ in \cite{Erdos1966}, and to Francesco
Cellarosi \cite{Cellarosi} for pointing out three gaps in an earlier version of the
paper.

\end{document}